\documentclass[11pt,a4paper]{amsart}
\usepackage{tikz}
\usepackage{amssymb}
\usepackage{amsthm}
\usepackage{url}

\newtheorem{theorem}{Theorem}[section]
\newtheorem{lemma}[theorem]{Lemma}
\newtheorem{corollary}[theorem]{Corollary}
\newtheorem{proposition}[theorem]{Proposition}
\newtheorem{conjecture}[theorem]{Conjecture}

\newtheorem{observation}[theorem]{Observation}

\DeclareMathOperator{\KG}{\mathrm{KG}}
\DeclareMathOperator{\SG}{\mathrm{SG}}

\newcommand{\FF}{\mathcal F}
\newcommand{\indep}[2]{\mathcal I_{#2}(#1)}

\tikzstyle{vertex}=[circle, fill=black, minimum size=6pt, inner sep=0]

\title{The chromatic number of $G$-stable Kneser graphs}

\author{Carmen Arana}
\thanks{The first author was partially supported by the Barrande Fellowship Program}
\address{Université de Paris, CNRS, IRIF, 75006 Paris, France.}
\email{arana@irif.fr}

\author{Penny Haxell}
\thanks{The second author was partially supported by NSERC}
\address{Department of Combinatorics and Optimization, University of Waterloo, Waterloo ON Canada N2L 3G1.}
\email{pehaxell@uwaterloo.ca}

\author{Tom\'a\v s Kaiser}
\address{Department of Mathematics and European Centre of Excellence
  NTIS (New Technologies for the Information Society), University of
  West Bohemia, Pilsen, Czech Republic.}
\email{kaisert@kma.zcu.cz}

\author{Mat\v ej Stehl\'ik}
\thanks{The fourth author was partially supported by an Emergence en recherche IDEX grant from Université Paris Cité.}
\address{Université de Paris, CNRS, IRIF, 75006 Paris, France.}
\email{matej@irif.fr}

\begin{document}
\begin{abstract}
  Extending the notions of Kneser graphs and Schrijver graphs, we define the \emph{$G$-stable Kneser graph} $\KG(G,k)$ (where $k\geq 2$ and $G$ is a graph on $n\geq 2k$ vertices) as the graph whose vertices are the $k$-element independent sets in $G$ and whose edges join pairs of disjoint independent sets. The well-known result of Lov\'{a}sz on the chromatic number of Kneser graphs implies that $\chi(\KG(G,k)) \leq n-2k+2$ for any $G$. We call $G$ $k$-tight if equality holds, and we study the following question: Which graphs are $k$-tight? A result of Daneshpajouh et al.\ implies a characterization of $2$-tight graphs. We prove a necessary condition for a graph to be $k$-tight for $k\geq 3$ which extends this characterization, and we conjecture it to be sufficient. Another result gives a lower bound to the chromatic number of the $G$-stable Kneser graphs when $G$ is a disjoint union of cycles. We use this result to give a short proof of the fact that $4$-regular graphs on $3k$ vertices consisting of two vertex-disjoint cycles and $k$ cycles of length $3$ admit an independent set of size $k$.
\end{abstract}

\maketitle

\section{Introduction}

Recall that the Kneser graph $\KG(n,k)$, where $n\geq 2k\geq 2$, is the graph whose vertices are the $k$-subsets of $[n]=\{1,\dots,n\}$ and whose edges join disjoint $k$-subsets. In 1955 Kneser~\cite{Kne55} observed that $\KG(n,k)$ can be colored with $n-2k+2$ colors, and conjectured that this is optimal. This became known as \emph{Kneser's conjecture}, and was proved by Lov\'asz~\cite{Lov78} in 1978 using topological methods. In the same year, Schrijver~\cite{Sch78} showed that the subgraph $\SG(n,k)$ induced by the $k$-sets that are independent in the $n$-cycle $C_n$ (with vertices numbered by elements of $[n]$ in cyclic order) still has chromatic number $n-2k+2$. Schrijver called these graphs \emph{stable Kneser graphs}; they are also widely known as \emph{Schrijver graphs}.

To express the above results within a more general framework, we adopt the following notion used, e.g., in~\cite[Section~3.3]{Mat03} or~\cite{DMM21}. The \emph{Kneser graph} $\KG(\FF)$ of a hypergraph $\FF$ has the hyperedges of $\FF$ as vertices, and has edges joining disjoint hyperedges. For a graph $G$ and integer $k\geq 2$, let $\indep G k$ be the hypergraph on the vertex of set of $G$ whose hyperedges are the independent $k$-sets of $G$. We define the \emph{$G$-stable Kneser graph} $\KG(G,k)$ as
\begin{equation*}
  \KG(G,k) = \KG(\indep G k).
\end{equation*}

Then $\SG(n,k)$ is precisely $\KG(C_n, k)$, and $\KG(n,k)$ can likewise be obtained as $\KG(\overline{K_n},k)$, where $\overline{K_n}$ is the edgeless graph on $n$ vertices. As another example, choosing $G$ as the $(s-1)$-st power of $C_n$ (where $s\geq 3$) yields the \emph{$s$-stable Kneser graphs} studied in~\cite{Meu11} and in~\cite{Che15,Jon12}.

It is easy to see that if $G'$ is a spanning subgraph of $G$, then $\KG(G,k)$ is an induced subgraph of $\KG(G',k)$. By taking $G'$ to be $\overline{K_n}$, it follows that for any $n$-vertex graph $G$ and for $k\leq\frac n2$, the chromatic number of $\KG(G,k)$ is at most $n-2k+2$. The theorem of Schrijver says that $G=C_n$ achieves this maximum possible chromatic number. By the above, all spanning subgraphs of $C_n$ achieve it as well. Are there other $n$-vertex graphs with this property for some $k$?

This is the main problem we study in this paper. For brevity, we call an $n$-vertex graph $G$ \emph{$k$-tight} if $\chi(\KG(G,k)) = n-2k+2$. Thus, $C_n$ and all its spanning subgraphs are $k$-tight for any $k\leq \frac n2$. Furthermore, $k$-tight graphs on $2k$ vertices are easy to characterize: they are the graphs whose vertex set can be partitioned into two independent $k$-sets, i.e., all subgraphs of the complete bipartite graph $K_{k,k}$. In what follows, we can therefore focus on $k$-tight graphs on at least $2k+1$ vertices.

Our research is inspired by \cite{DMM21} whose subject matter is the chromatic number of $\KG(\FF)$ as defined above, in the case where the hypergraph $\FF$ is actually a graph $G$ (i.e., its hyperedges have size $2$). In our notation, $\KG(\FF)$ is then just $\KG(\overline G,2)$, where the bar stands for complement. By the above theorem of Lov\'{a}sz, the chromatic number of this graph is $n-2$, and the main result of \cite{DMM21} is a characterization of the graphs $G$ for which this value is achieved. It translates to the following characterization of $2$-tight graphs:
\begin{theorem}\label{t:dmm}
  A graph on at least $5$ vertices is $2$-tight if an only if it is a disjoint union of paths and of cycles of length at least $5$. 
\end{theorem}

Our main result proves a corresponding necessary condition for every $k \geq 2$:

\begin{theorem}\label{thm:main}
  For $k\geq 2$, if a graph $G$ on at least $2k+1$ vertices contains a vertex of degree at least $3$, or properly contains a cycle of length at most $4k-4$, then $G$ is not $k$-tight.
\end{theorem}

In particular, a $k$-tight graph is a disjoint union of paths and cycles, and every cycle it properly contains has length at least $4k-3$. We conjecture that the converse holds as well, which would extend Theorem~\ref{t:dmm} to all $k \geq 2$.

\begin{conjecture}\label{conj:main}
  For any $k\geq 2$, a graph $G$ on at least $2k+1$ vertices is $k$-tight if and only if either $G$ is a cycle, or every component of $G$ is a path or a cycle of length at least $4k-3$.
\end{conjecture}

For $k=2$, this is Theorem~\ref{t:dmm}. For $k \geq 3$, the conjecture is consistent with our computational tests, but even the simplest case not covered by Schrijver's theorem, a cycle of length at least $4k-3$ together with an isolated vertex, is open; we were only able to prove it asymptotically.

\begin{theorem}\label{thm:cycle+dot}
  For all $k \geq 2$ there exists $n(k)$ such that for all $n \geq n(k)$, the graph $C_{n-1} \sqcup K_1$ is $k$-tight.
\end{theorem}

Interestingly, together with Theorem~\ref{thm:main} this shows that $C_n$ is not the only maximal $k$-tight graph: for $n \geq n(k)$, so is $C_{n-1} \sqcup K_1$, since adding any edge to it creates a vertex of degree~$3$.

Theorem~\ref{thm:main} shows that a disjoint union of at least two cycles, one of which has length at most $4k-4$, is not $k$-tight, but says nothing about how far $\chi(\KG(G,k))$ falls below $n-2k+2$. In Section~\ref{sec:gen_lower_bd} we bound it from below, with no restriction on the cycle lengths.

\begin{theorem}\label{thm:chi_bound}
  Let $c,k,n$ be integers with $c \geq 1$, $k \geq 2$ and $n \geq 2k+c-1$. If $G$ is a disjoint union of $c$ cycles of total length $n$, then $\chi(\KG(G,k)) \geq n-2k+3-c$.
\end{theorem}

The bound is sharp: for each $c \leq k$ there is a disjoint union of $c$ cycles attaining it. In Section~\ref{sec:cycles+triangles} we use Theorem~\ref{thm:chi_bound} in place of Schrijver's theorem in an argument of Aharoni et al.~\cite{AABCKLZ17} to obtain a short proof of the extension of a conjecture of Du, Hsu and Hwang~\cite{DHH93} to disjoint unions of two cycles.

\section{Explicit colorings}\label{sec:proof_main_thm}

This section is devoted to proving Theorem~\ref{thm:main}.

First, we give two simple statements that allow us to reduce our problem to subgraphs. For a vertex $v$ of $G$, the \emph{star centered at $v$} is the set of all vertices of $\KG(G,k)$ that contain $v$. Any two of them meet in $v$, so a star is an independent set of $\KG(G,k)$; consequently, if the stars at the vertices of some $S \subseteq V(G)$ cover $\KG(G,k)$, then $\chi(\KG(G,k)) \leq |S|$.

\begin{lemma}\label{lem:subgraph}
  If $H \subseteq G$, then $\chi(\KG(G,k)) \leq \chi(\KG(H,k))+\lvert V(G)\setminus V(H)\rvert$.
\end{lemma}
  
\begin{proof}
  The vertices of $\KG(G,k)$, viewed as independent $k$-sets of $G$, can be partitioned into two sets: those containing at least one vertex of $V(G) \setminus V(H)$, and those contained entirely in $H$. The first set is covered by the stars at the vertices of $V(G) \setminus V(H)$, so it can be colored with $|V(G)\setminus V(H)|$ colors. Those in the second set are vertices of $\KG(H,k)$, so can be colored using $\chi(\KG(H,k))$ colors.
\end{proof}

\begin{corollary}\label{cor:monotone}
  The family of $k$-tight graphs is closed with respect to subgraphs; that is, $k$-tightness is a monotone property.
\end{corollary}

Next, we treat the case in which the maximum degree is larger than $2$. Recall that a \emph{claw} is a copy of $K_{1,3}$; we call its vertex of degree three the \emph{root}, and its three vertices of degree one the \emph{leaves}.

\begin{lemma} \label{lem:claw}
  Let $G$ be the disjoint union of a claw and $2k-3$ isolated vertices, for $k \in \mathbb{N}$. Then $\chi(\KG(G,k)) \leq 2$.
\end{lemma}
  
\begin{proof}
  We partition the vertices of $\KG(G,k)$ into the following color classes: 
  \begin{enumerate}
    \item $k$-sets containing two of the leaves of the claw,
    \item all other $k$-sets.
  \end{enumerate}
  
  Let $U,W$ be vertices of $\KG(G,k)$ containing two of the leaves of the claw. By the pigeonhole principle, $U \cap W \neq \varnothing$.
  
  Let $U,W$ be vertices of $\KG(G,k)$ containing at most one of the leaves of the claw. They can only contain one vertex of the claw, as containing the root is incompatible with containing a leaf. Therefore they must have $k-1$ vertices in the $2k-3$ vertices of $G$ not in the claw. Therefore by the pigeonhole principle, $U \cap W \neq \varnothing$.
\end{proof}

\begin{corollary} \label{cor:degree-3}
  If $G$ has $n \geq 2k+1$ vertices, and contains a vertex of degree at least $3$, then $\chi(\KG(G,k)) \leq n-2k+1$.
\end{corollary}

\begin{proof}
  Let $H$ be a subgraph of $G$ containing a vertex of degree at least 3, three of its neighbors and $2k-3$ other vertices. By Lemma~\ref{lem:claw}, $\chi(\KG(H,k)) \leq 2$. Therefore, by Lemma~\ref{lem:subgraph},
  \[
    \chi(\KG(G,k)) \leq 2+ \lvert V(G) \setminus V(H) \rvert=2+n-(2k+1)=n-2k+1,
  \]
  as required.
\end{proof}

In the remainder of this section we focus on the case where $G$ contains a short cycle. In the first two lemmas the cycle has length at most $2k$. We use the notation $[a,b]$, where $a$ and $b$ are integers, to denote the set $\{a,a+1,\dots,b\}$.

\begin{lemma} \label{lem:odd_cycle}
  Let $G$ be the disjoint union of a $(2k-2j-1)$-cycle and $2j+1$ isolated vertices, for some $j \in [0, k-2]$. Then $\chi(\KG(G,k)) \leq 1$.
\end{lemma}
  
\begin{proof}
  Any independent $k$-set of $G$ can have at most $k-j-1$ vertices in the cycle and therefore must contain at least $j+1$ of the isolated vertices. By the pigeonhole principle, any two such sets intersect, so $\KG(G,k)$ is edge-free.
\end{proof}

\begin{lemma} \label{lem:even_cycle}
  Let $G$ be the disjoint union of a $(2k-2j)$-cycle and $2j+1$ isolated vertices, for some $j \in [0, k-2]$. Then $\chi(\KG(G,k)) \leq 2$.
\end{lemma}

\begin{proof}
  The two color classes are as follows:
  \begin{enumerate}
    \item $k$-sets containing an odd-labeled vertex of the cycle.
    \item all other $k$-sets.
  \end{enumerate}
    
  Any independent set of $G$ has at most $k-j$ vertices in the cycle. A set with fewer vertices on the cycle contains at least $j+1$ of the isolated vertices, so by the pigeonhole principle, any two such sets must intersect. Any independent $k$-set that contains $k-j$ vertices in the cycle and an odd-labeled vertex must contain all odd vertices of the cycle, therefore it intersects with all vertices in the color class. 

  For the second class, suppose for a contradiction that $U$ and $W$ are disjoint independent $k$-sets containing no odd-labeled vertex of the cycle. The union $U \cup W$ contains at most $k-j$ vertices of the cycle, and at most $2j+1$ isolated vertices, so $2k = |U|+|W| \leq (k-j)+(2j+1) = k+j+1$, contradicting $j \leq k+2$ and completing the proof.
\end{proof}

The remaining cycle lengths are harder, and are handled by the following two theorems. We will see in Corollary~\ref{cor:short_cycles} that they imply that $G$ is not $k$-tight if it properly contains any cycle of length between $2k+1$ and $4k-4$.

\begin{theorem} \label{thm:(4k-4)-cycle}
  For each $k\geq 2$, $\chi(\KG(C_{4k-4}\sqcup K_1,k)) \leq 2k-2$.
\end{theorem}

\begin{theorem} \label{thm:(4k-5)-cycle}
  For each $k\geq 2$, $\chi(\KG(C_{4k-5}\sqcup K_1,k)) \leq 2k-3$.
\end{theorem}

\subsection{Proof of Theorem \ref{thm:(4k-4)-cycle}}
For $k \geq 2$, let $G_\ell=C_{4\ell-4} \sqcup K_1$, label the vertices of the cycle $0, \ldots, 4\ell-5$, and label the isolated vertex $\infty$. A vertex $v$ of the cycle is a \emph{center} if $v \equiv 2$ or $3 \pmod 4$. Note that there are exactly $2\ell-2$ centers. The $2\ell-2$ non-centers in the cycle induce a perfect matching of size $\ell-1$, so every independent $\ell$-set of $C_{4\ell-4}$ must contain at least one center. This shows that the $2k-2$ stars centered at all vertices $v$ such that $v \equiv 2$ or $3 \pmod 4$ cover all vertices of $\SG(4k-4,k)$. Our goal is to modify this cover by replacing certain $k$-sets in these stars by $k$-sets containing $\infty$.

For a center $v$ of $G_\ell$ such that $v \equiv 2 \pmod 4$, set
\[
  N_\ell(v) = \{0,\,v-2,\,\infty\} \cup \left(\{v+3\} \cap V(G_{\ell})\right),
\]
and for a center $v$ of $G_\ell$ such that $v \equiv 3 \pmod 4$, set
\[
  N_\ell(v) = \{1,\,\infty\} \cup \{4j: 0<4j<v\} \cup \left(\{v+2\}\cap V(G_\ell)\right).
\]
Let $H_{\ell}$ be the graph obtained from $G_{\ell}$ by adding edges from each center $v$ to the set $N_{\ell}(v)$; see Figure~\ref{fig:H4} for an illustration. Note that $G_{\ell}$ is a spanning subgraph of $H_{\ell}$.

Let $\mathcal C_{\ell}(v)$ consist of the set of all vertices $U$ of $\KG(G_\ell,\ell)$ such that either $v \in U$ and $U$ intersects $N_{\ell}(v)$, or $N_{\ell}(v) \subseteq U$. Note that $\mathcal C_{\ell}(v)$ forms an intersecting family, and therefore an independent set in $\KG(G_{\ell},\ell)$, for every center $v$ of $G_{\ell}$.

\begin{observation} \label{obs:centers}
  Let $k \geq 3$. A vertex $v \in V(G_{k-1})$ is a center of $G_{k-1}$ if and only if $v$ is a center of $G_k$ and $v \leq 4k-9$. Moreover,
  \[
    N_k(v) =%
    \begin{cases}
      N_{k-1}(v) \cup \{4k-7\} & \text{if } v \in \{4k-10,4k-9\},\\
      N_{k-1}(v) & \text{otherwise.}
    \end{cases}
  \]
\end{observation}

The vertices of $\KG(G_k,k)$, seen as independent $k$-sets of $G_k$, can be partitioned into three types:
\begin{enumerate}
  \item independent $k$-sets not containing $\infty$;
  \item independent $k$-sets containing $\infty$ and a vertex $v \equiv 2$ or $3 \pmod 4$;
  \item independent $k$-sets containing $\infty$ and no vertex $v \equiv 2$ or $3 \pmod 4$.
\end{enumerate}

Our goal is to show that all the vertices of $\KG(G_k,k)$ are in some $\mathcal{C}_k(v)$. The vertices of type 2 contain $\{v, \infty\}$ for some $v \equiv 2,3 \pmod 4$, so they are in $\mathcal{C}_k(v)$. It remains to show that vertices of type 1 and 3 are in some $\mathcal{C}_k(v)$, which we will show in the next two lemmas.

The fact that every vertex of type 1 is in some $\mathcal{C}_k(v)$ can be expressed in terms of (the non-existence of) independent sets of size $k$ in $H_{\ell}-\infty$. Indeed, if $U$ is an independent $k$-set in $G_k-\infty$, but not in $H_k-\infty$, then $U$ contains some center $v$, as well as a vertex in $N(v)$. Hence, $U \in \mathcal C_k(v)$.

\begin{figure}
  \definecolor{myred}    {RGB}{196, 62, 62}
  \definecolor{myorange} {RGB}{228,145, 35}
  \definecolor{mygreen}  {RGB}{ 55,140, 75}
  \definecolor{myblue}   {RGB}{ 55,105,205}
  \definecolor{mypurple} {RGB}{130, 80,170}
  \definecolor{myteal}   {RGB}{ 30,150,155}
  \begin{tikzpicture}
    \node[vertex]          (0)  at ( 1.2, 5) {};
    \node[vertex]          (1)  at (-1.2, 5) {};
    \node[vertex,mypurple] (2)  at ( 1.2, 4) {};
    \node[vertex,myteal]   (3)  at (-1.2, 4) {};
    \node[vertex]          (4)  at (-1.2, 3) {};
    \node[vertex]          (5)  at ( 1.2, 3) {};
    \node[vertex,myorange]          (6)  at ( 1.2, 2) {};
    \node[vertex,mygreen]          (7)  at (-1.2, 2) {};
    \node[vertex]          (8)  at (-1.2, 1) {};
    \node[vertex]          (9)  at ( 1.2, 1) {};
    \node[vertex,myblue]          (10) at ( 1.2, 0) {};
    \node[vertex,myred]          (11) at (-1.2, 0) {};

    \path[thick,dotted]
      (0) edge (1)
      (1) edge (2)
      (2) edge (3)
      (3) edge (4)
      (4) edge (5)
      (5) edge (6)
      (6) edge (7)
      (7) edge (8)
      (8) edge (9)
      (9) edge (10)
      (10) edge (11)
      (11) edge (0);

    \path
      (0) edge[ultra thick,mypurple] (2)
      (2) edge[ultra thick,mypurple] (5);

    \path
      (1) edge[ultra thick,myteal] (3)
      (3) edge[ultra thick,myteal] (5);

    \path
      (4) edge[ultra thick,myorange] (6)
      (6) edge[ultra thick,myorange] (9)
      (6) edge[ultra thick,myorange, bend right] (0);

    \path
      (4) edge[ultra thick,mygreen] (7)
      (7) edge[ultra thick,mygreen] (9)
      (7) edge[ultra thick,mygreen,bend left] (1);

    \path
      (8) edge[ultra thick,myblue] (10)
      (10) edge[ultra thick,myblue,bend right,out=-45,in=-135] (0);

    \path
      (8) edge[ultra thick,myred] (11)
      (4) edge[ultra thick,myred,bend right] (11)
      (11) edge[ultra thick,myred,bend left,out=45,in=135] (1);

    \foreach \v/\p in {
      0/above,
      1/above,
      2/right,
      3/left,
      4/left,
      5/right,
      6/right,
      7/left,
      8/left,
      9/right,
      10/below,
      11/below}
      \node[\p=2pt] at (\v) {\scriptsize $\v$};

    \begin{scope}[xshift=6cm]

      \node[vertex] (2)  at ( 1.2, 4) {};
      \node[vertex] (3)  at (-1.2, 4) {};
      \node[vertex] (5)  at ( 1.2, 3) {};
      \node[vertex] (6)  at ( 1.2, 2) {};
      \node[vertex] (7)  at (-1.2, 2) {};
      \node[vertex] (9)  at ( 1.2, 1) {};

      \path
        (2) edge[semithick] (3)
        (5) edge[semithick] (6)
        (6) edge[semithick] (7)
        (2) edge[semithick] (5)
        (3) edge[semithick] (5)
        (6) edge[semithick] (9)
        (7) edge[semithick] (9);

      \foreach \v/\p in {
        2/right,
        3/left,
        5/right,
        6/right,
        7/left,
        9/right}
        \node[\p=2pt] at (\v) {\scriptsize $\v$};

    \end{scope}
  \end{tikzpicture}

  \caption{Left: the graph $H_4-\infty$, with the $12$-cycle dotted; the remaining edges have the color of the incident center. Right: the subgraph $H_5[T]$ used in the proof of Lemma~\ref{lem:independence}.}
  \label{fig:H4}

\end{figure}
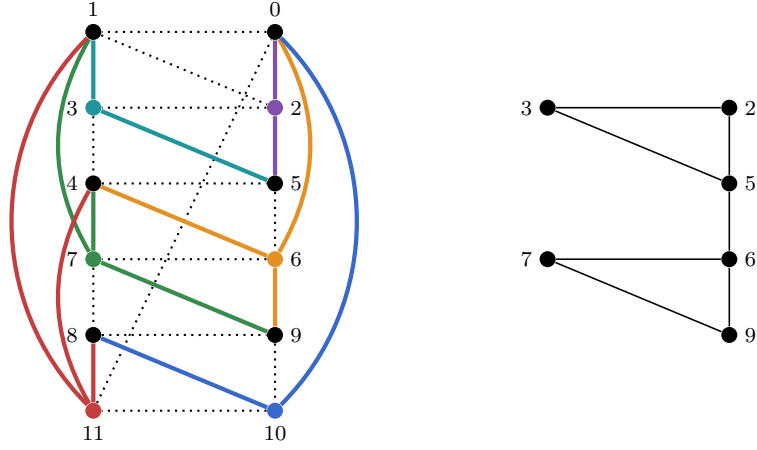
    
\begin{lemma}\label{lem:independence}
  The graph $H_k-\infty$ contains no independent $k$-set.
\end{lemma}

\begin{proof}
  We proceed by induction on $k$. For $k=2$, the graph $H_2-\infty$ is the complete graph on $4$ vertices, which contains no independent $2$-set.

  Assume $k \geq 3$ and let $A$ be an independent set in $H_k-\infty$. By Observation~\ref{obs:centers},
  \[
    H_{k-1} = H_k[V(G_{k-1})] + \{0,4k-9\},
  \]
  so $A \cap V(G_{k-1}-\infty)$ is independent in $H_{k-1}-\infty$, unless it contains both $0$ and $4k-9$.

  If $A$ contains both $0$ and $4k-9$, then $A \cap [4k-8,4k-5] = \emptyset$, and since $A \setminus \{4k-9\}$ is an independent set in $H_{k-1}-\infty$, we have $|A|-1 \leq k-2$ by the induction hypothesis, whence $|A| \leq k-1$.

  If $A \cap V(G_{k-1}-\infty)$ is independent in $H_{k-1}-\infty$, then $|A \cap V(H_{k-1}-\infty)| \leq k-2$ by the induction hypothesis. If $|A \cap [4k-8,4k-5]| \leq 1$, then $|A| \leq k-1$, as required. If $|A \cap [4k-8,4k-5]| \geq 2$, then $A$ contains both $4k-7$ and $4k-5$. The vertex $4k-7$ is adjacent to $4k-9$ and $4k-10$, and $4k-5$ is adjacent to $0$, $1$, as well as all $4j$ for $0 < 4j < 4k-5$. Therefore,
  \[
    A \cap V(G_{k-1}-\infty) \subseteq T:= \{j : 2 \leq j \leq 4k-11, j \not \equiv 0 \pmod 4\}.
  \]
  The graph $H_k[T]$ contains a disjoint union of $k-3$ triangles as a spanning subgraph (see Figure~\ref{fig:H4}); therefore, $|A \cap V(G_{k-1}-\infty)| = |A \cap T| \leq k-3$, so $|A| \leq k-1$ again.
\end{proof}

The following lemma shows that vertices of the third type are in some $\mathcal{C}_k(v)$.

\begin{lemma}
  Let $k \geq 2$ and let $B$ be an independent $(k-1)$-set of $G_k-\infty$ containing no center. Then $B \cup \{\infty\} \supseteq N_k(v)$ for some center $v$.
\end{lemma}

\begin{proof}
  First, observe that the non-centers induce a perfect matching of size $k-1$, formed by the edges $\{4i,4i+1\}$ with $0 \leq i \leq k-2$. Since $|B|=k-1$, $B$ must contain exactly one vertex from each edge of the matching. In particular, $B$ contains exactly one of $4k-8$ and $4k-7$.

  We will prove the lemma by induction. For $k=2$ there are only two possibilities, namely $B = \{0\}$ and $B = \{1\}$, and $N_2(2) = \{0,\infty\}$ while $N_2(3) = \{1,\infty\}$.

  Assume $k \geq 3$, and set $B' = B \cap [0,4k-9]$. By the first paragraph, we have $|B'| = k-2$. The vertex $4k-9$ is a center, so $B'$ is an independent set of the cycle $C_{4k-8}$ in $G_{k-1}$, containing no center of $G_{k-1}$. By the induction hypothesis, there is a center $v$ of $G_{k-1}$ such that $B' \cup \{\infty\} \supseteq N_{k-1}(v)$.

  By Observation~\ref{obs:centers}, it suffices to consider three cases. If $v \leq 4k-13$, then $N_k(v) = N_{k-1}(v) \subseteq B \cup \{\infty\}$. If $v \in \{4k-10, 4k-9\}$ and $4k-7 \in B$, then $N_k(v) = N_{k-1}(v) \cup \{4k-7\} \subseteq B \cup \{\infty\}$.

  Finally, suppose that $v \in \{4k-10, 4k-9\}$ and $4k-7 \notin B$, so $4k-8 \in B$. If $v = 4k-10$, then $N_{k-1}(4k-10) = \{0,4k-12,\infty \} \subseteq B \cup \{\infty\}$, so $0 \in B$ and hence $N_k(4k-6) = \{0,4k-8,\infty\} \subseteq B \cup \{\infty\}$. If $v = 4k-9$, then $N_{k-1}(4k-9) = \{1,\infty\} \cup \{4j : 0 < 4j < 4k-9\} \subseteq B \cup \{\infty\}$, and therefore
  \begin{align*}
    N_k(4k-5) &= \{1,\infty\} \cup \{4j : 0 < 4j < 4k-5\}\\
              &= N_{k-1}(4k-9) \cup \{4k-8\} \subseteq B \cup \{\infty\},
  \end{align*}
  as required.
\end{proof}

This completes the proof of Theorem~\ref{thm:(4k-4)-cycle}.

\subsection{Proof of Theorem~\ref{thm:(4k-5)-cycle}}
For $k=2$ the graph $\KG(C_3 \sqcup K_1,2)$ has no edges, so its chromatic number is $1 = 2k-3$. We may therefore assume that $k \geq 3$. We will construct the required coloring of $\KG(C_{4k-5} \sqcup K_1, k)$ from the one constructed in the proof of Theorem~\ref{thm:(4k-4)-cycle}. 

Label the isolated vertex of the elements of $C_{4k-5} \sqcup K_1$ by $\infty$, and the vertices of $C_{4k-5}$ by $0,1,\star,4,5 \dots, 4k-5$ in cyclic order. We use the same notation for the color classes of $\KG(C_{4k-4} \sqcup K_1,k)$ as in the proof of Theorem~\ref{thm:(4k-4)-cycle}.

Since no independent set of $G_k$ contains both $2$ and $3$, the map
\[
  f(A) =
    \begin{cases}
      (A \setminus \{2,3\}) \cup \{\star\} & \text{ if } \{2,3\} \cap A \neq \varnothing \\
      A                                    & \text{ otherwise}
    \end{cases}
\]
sends every vertex of $\KG(G_k,k)$ to a $k$-subset of $V(C_{4k-5} \sqcup K_1)$.

For each center $v \geq 6$ of $G_k$, let $\mathcal K_k(v)$ consist of the vertices of $\KG(C_{4k-5} \sqcup K_1, k)$ in $f(\mathcal C_k(v))$, and let $\mathcal{K}_k(\star)$ be the independent $k$-sets of $C_{4k-5} \sqcup K_1$ containing at least two elements of $\{5,\star,\infty\}$. As $G_k$ has $2k-2$ centers, this is a family of size $2k-3$, so it suffices to show that each member of the family is intersecting and that together they cover $\KG(C_{4k-5} \sqcup K_1,k)$.

$\mathcal{K}_k(\star)$ forms an intersecting family by the pigeonhole principle. Next, let $A, B \in \mathcal{K}_k(v)$, where $6 \leq v \leq 4k-5$ and $v \equiv 2,3 \pmod 4$. If neither $A$ nor $B$ contains $\star$, then $A, B \in \mathcal{C}_k(v)$, so $A \cap B \neq \varnothing$. If both $A$ and $B$ contain $\star$, then $A \cap B \neq \emptyset$. Finally, if $A$ contains $\star$ but $B$ does not, then $B \in \mathcal{C}_k(v)$ and $A = f(A')$, for some $A' \in \mathcal{C}_k(v)$. Hence, $A' \cap B \neq \emptyset$, and since $B$ contains neither $2$ nor $3$, they intersect in a vertex of $A' \setminus \{2,3\} \subseteq A$. Similarly for the case where $A$ does not contain $\star$ and $B$ does.

To show that these classes form a covering, note that $N_k(2) = \{0,5,\infty\}$ and $N_k(3) = \{1,5,\infty\}$, and let $A$ be a vertex of $\KG(C_{4k-5} \sqcup K_1,k)$.

Suppose $\star \notin A$. Then $A$ is a vertex of $\KG(G_k,k)$ containing neither $2$ nor $3$, so $A \in \mathcal C_k(v)$ for some center $v$. If $v \geq 6$, then $f(A) = A$ and $A \in \mathcal K_k(v)$. If $v \in \{2,3\}$, then $v \notin A$, so by the definition of $\mathcal C_k(v)$, $N_k(v) \subseteq A$. In particular, $5, \infty \in A$ and $A \in \mathcal{K}_k(\star)$.

Now suppose $\star \in A$. If $A$ contains $5$ or $\infty$, then $A \in \mathcal{K}_k(\star)$. Otherwise set $A' = (A \setminus \{\star\}) \cup \{3\}$. Since $\star \in A$ we have $1,4 \notin A$, and $2 \notin V(C_{4k-5}\sqcup K_1)$. As the neighbors of $3$ in $G_k$ are $2$ and $4$, the set $A'$ is independent in $G_k$, and clearly $f(A') = A$. Now $2 \notin A'$ and $N_k(2) \not\subseteq A'$ because $5 \notin A'$, so $A' \notin \mathcal C_k(2)$ and $A'$ is disjoint from $N_k(3) = \{1,5,\infty\}$, so $A' \notin \mathcal C_k(3)$. Hence $A' \in \mathcal C_k(v)$ for some center $v \geq 6$, and $A = f(A') \in \mathcal K_k(v)$.

This completes the proof of Theorem~\ref{thm:(4k-5)-cycle}.

\subsection{Application to an upper bound}

\begin{lemma}\label{lem:homomorphism}
  For all integers $k\geq 2$ and $r\geq 2k-2$, there exists a homomorphism from $\KG(C_{r+2} \sqcup {K_1},k)$ to $\KG(C_{r} \sqcup K_1,k-1)$.
\end{lemma}
      
\begin{proof}
  In each of the two graphs under consideration, we label the vertex of $K_1$ with 0, and the vertices of $C_{r+2}$ and $C_{r}$ with $[r+2]$ and $[r]$, respectively, in the natural order. We define $f:\KG(C_{r+2}\sqcup K_1,k)\to\KG(C_r\sqcup K_1,k-1)$ as follows. For each vertex $A$ of $\KG(C_{r+2}\sqcup K_1,k)$, we form $f(A)$ by removing the largest element of $A$. Clearly $|f(A)|=k-1$.
  
  First note that $\{r+1,r+2\}\cap f(A)=\emptyset$, which shows $f(A)\subseteq V(C_r\sqcup K_1)$. To see that $f(A)$ is independent in $C_r\sqcup K_1$, observe that since $A$ is independent in $C_{r+2}\sqcup K_1$, the only consecutive elements on $C_r$ that could both be in $f(A)$ are 1 and $r$. But then neither $r+1$ nor $r+2$ can be in $A$, making $r$ the largest element of $A$. This implies $r \notin f(A)$. Therefore $f(A)$ is a vertex of $\KG(C_r\sqcup K_1, k-1)$.

  Finally, note that if $AB$ is an edge of $\KG(C_{r+2}\sqcup K_1,k)$ then $A\cap B=\emptyset$, so by definition $f(A)\cap f(B)=\emptyset$ as well.
\end{proof}

\begin{corollary}\label{cor:med_cycle}
  For all integers $k, t$ such that $k\geq 2$ and $2k+1 \leq t \leq 4k-4$, we have $\chi(\KG(C_{t}\sqcup K_1,k)) \leq n-2k+1$, where $n=|V(C_{t}\sqcup K_1)|=t+1$.
\end{corollary}

\begin{proof}
  Put $m = \lfloor 2k-2-t/2 \rfloor$, and observe that $0 \leq m \leq k-3$, and therefore $k-m \geq 3$ and $t-2m \geq 7$. We can thus apply Lemma~\ref{lem:homomorphism} $m$ times to obtain a homomorphism from $\KG(C_t \sqcup K_1,k)$ to $\KG(C_{t-2m}\sqcup K_1,k-m)$. In particular, $\chi(\KG(C_{t}\sqcup K_1,k)) \leq \chi(\KG(C_{t-2m}\sqcup K_1,k-m))$.

  If $t$ is even, then $t+2m = 4k-4$, so $t-2m = 4(k-m)-4$, and by Theorem~\ref{thm:(4k-4)-cycle},
  \[
    \chi(\KG(C_{t-2m}\sqcup K_1,k-m)) \leq 2(k-m)-2 = t-2k+2.
  \]
  Similarly, if $t$ is odd, then $t+2m = 4k-5$, so $t-2m = 4(k-m)-5$, and by Theorem~\ref{thm:(4k-5)-cycle},
  \[
    \chi(\KG(C_{t-2m}\sqcup K_1,k-m)) \leq 2(k-m)-3 = t-2k+2,
  \]
  as required.
\end{proof}

We summarize our results about short cycles as follows.

\begin{corollary} \label{cor:short_cycles}
  Let $k \geq 2$ and let $G$ be a graph on $n \geq 2k+1$ vertices. If $G$ contains a cycle of length at most $4k-4$ and $G$ is not itself a cycle, then $\chi(\KG(G,k)) \leq n-2k+1$.
\end{corollary}

\begin{proof}
  First suppose that $G$ contains a cycle $C_t$ with $t<2k+1$ vertices. Let $H=C_t\sqcup U$ where $U$ consists of $2k+1-t$ vertices if $t$ is even, and $2k-t$ isolated vertices if $t$ is odd. Note that $H$ is a subgraph of $G$. If $t$ is even then by Lemma~\ref{lem:even_cycle} we find $\chi(\KG(H,k))=2=|V(H)| -2k+1$.  If $t$ is odd then by Lemma~\ref{lem:odd_cycle} we get $\chi(\KG(H,k))= 1 =|V(H)| -2k+1$.
    Hence in either case $H$ is not $k$-tight.

  Finally suppose $G$ contains a cycle $C_t$ where $2k+1\leq t\leq 4k-4$. Define $H=C_t\sqcup K_1$, which is a subgraph of $G$ by assumption. Then Corollary~\ref{cor:med_cycle} tells us that $H$ is not $k$-tight. 

  Thus in all cases,  Corollary~\ref{cor:monotone} confirms that $G$ is not $k$-tight, as required.
\end{proof}

Theorem~\ref{thm:main} now follows from Corollaries~\ref{cor:degree-3} and~\ref{cor:short_cycles}.

\section{Asymptotic lower bound on chromatic number}\label{sec:asymptotic}

Call an independent set $\mathcal A$ in $\SG(n,k)$ \emph{anchored} if some vertex
$U \in \mathcal A$, called an \emph{anchor}, satisfies $|U \cap W| \geq 2$ for all
$W \in \mathcal A$. Given integers $k$ and $n$, let $a(n,k)$ be the maximum size of
an anchored independent set in $\SG(n,k)$.

\begin{lemma}\label{lem:anchor}
  For every $k \geq 2$ and $n\geq 2k+1$ we have 
  \begin{equation}
    \label{eq:1}
    a(n,k)\leq\sum_{j=2}^k(j-1)\binom{n-j}{k-2}.
  \end{equation}
\end{lemma}

\begin{proof}
  Let $\mathcal A$ be anchored with anchor $U$. Label the elements of $U$ by $[k]$. Then since each $W \in \mathcal A$ satisfies
  $|W\cap U|\geq 2$, we may define $p(W)$ to be the lexicographically smallest pair of elements of $[k]$ contained in $W$. Then for each pair $\{i,j\}$ with $1\leq i< j\leq k$, we have 
  \[
    |\{W\in{\mathcal A}:p(W)=\{i,j\}\}|\leq \binom{n-j}{k-2}.
  \] 
  Since for each $2 \leq j\leq k$ there are $j-1$ elements $i$ with $1\leq i<j$, we obtain the inequality~\eqref{eq:1}.
\end{proof}

In order to show Theorem~\ref{thm:cycle+dot} we show the following result on optimal colorings of Schrijver graphs. Below we use the basic fact that $|V(\SG(n,k)| = \tfrac{n}{k}\tbinom{n-k-1}{k-1}$ (see e.g.~\cite{Hav24,Kap43}). Note that for $n\geq k^2+k$,
\begin{equation}
  \label{eq:2}
  \frac n k \binom{n-k-1}{k-1} \geq \frac n k\left(\frac{n-k-1}{k-1}\right)^{k-1} \geq \frac{n^k}{k^k} 
\end{equation}
using the standard lower bound to the binomial coefficient in the first inequality, and the assumption $n\geq k^2+k$ in the second. 

\begin{theorem}\label{thm:rainbow}
  For every $k \geq 2$, there exists $n_0(k)$ such that for every $n \geq n_0(k)$ the following holds. For each optimal coloring of $\SG(n,k)$, there exists a stable $(k-1)$-set $F$ in the cycle $C_n$ such that all $n-2k+2$ colors appear on the subgraph $G_F$ induced by the stable $k$-sets in $C_n$ that are disjoint from $F$.
\end{theorem}

\begin{proof}
  Let $k\geq 2$ be fixed. We claim that $a(n,k) \leq k^2n^{k-2}$. Indeed, the binomial coefficient at the right hand side of~\eqref{eq:1} is clearly at most $n^{k-2}$ for any $j$, which implies the stated bound.

  By comparing this with~\eqref{eq:2}, it follows that there is a minimum integer $n_0(k)$ such that for every $n \geq n_0(k)$,
  \[
    (n-2k+2) \cdot a(n,k) < |V(\SG(n,k))| = \frac{n}{k}\binom{n-k-1}{k-1}.
  \]

  Now fix $n \geq n_0(k)$, and fix an optimal coloring of $\SG(n,k)$. By the above inequality, there must be at least one color class $\mathcal A$ that is non-anchored. Observe that there must be a vertex $U \in \mathcal A$ that is adjacent to at least one vertex of every color different from $\mathcal{A}$, otherwise we could recolor the vertices of $\mathcal A$ by other colors and obtain
  an $(n-2k+1)$-coloring of $\SG(n,k)$, which is impossible.

  Since $\mathcal A$ has no anchor, there exists a vertex $W \in \mathcal A$ such that $|U \cap W| = 1$. Set $F = U \setminus W$; since $F$ is a $(k-1)$-element subset of the stable set $U$, it is itself stable. Every neighbor $X$ of $U$ in $\SG(n,k)$ satisfies $X \cap U = \emptyset$, so a fortiori, $F \cap X = \emptyset$. Moreover, $F \cap W = (U\setminus W) \cap W = \emptyset$. Hence, the subgraph $G_F$ contains all neighbors of $U$, plus the vertex $W$; it spans all $n-2k+2$ colors.
\end{proof}

\begin{proof}[Proof of Theorem~\ref{thm:cycle+dot}]
  Set $n(k) = n_0(k)+1$, where $n_0(k)$ is as in Theorem~\ref{thm:rainbow}, and let $n \geq n(k)$. Let $H=C_{n-1} \sqcup K_1$, where the isolated vertex is labeled $0$.

  Suppose $\KG(H,k)$ has a proper $(n-2k+1)$-coloring $c$. The stable $k$-tuples of $C_{n-1}$ induce a subgraph of $\KG(H,k)$ isomorphic to $\SG(n-1,k)$, which is $(n-2k+1)$-chromatic.

  Therefore, the restriction of $c$ to $\SG(n-1,k)$ is an optimal coloring of $\SG(n-1,k)$. By Theorem~\ref{thm:rainbow} applied to $\SG(n-1,k)$, there exists a stable $(k-1)$-tuple $F$ of $C_{n-1}$ such that the vertices of $\SG(n-1,k)$ disjoint from $F$ span all $n-2k+1$ colors. But then the vertex $F \cup \{0\}$ cannot be colored; a contradiction. Hence, $\chi(\KG(H,k))\geq n-2k+2$.
\end{proof}

\section{General lower bound}\label{sec:gen_lower_bd}

Conjecture~\ref{conj:main} asserts that disjoint unions of cycles of length at least $4k-3$ are $k$-tight. What if we allow shorter cycles? In this section we prove a lower bound on $\chi(\KG(G,k))$ for $G$ a disjoint union of $c$ cycles, with no restriction on their lengths, and show that it is best possible.

The following lemma plays a key role in the proof of the lower bound.

\begin{figure}
  \centering
  \begin{tikzpicture}[scale=1.5, every label/.style={label distance=6pt,anchor=center}]
    \foreach \i/\lab in {1/1, 2/2, 3/3, 4/a, 5/{a+1}, 6/6, 7/7, 8/m} {
      \pgfmathsetmacro{\ang}{90-45*(\i-1)}
      \node[vertex,label={\ang:{\scriptsize $\lab$}}] (\i) at (\ang:1) {};
    }
    \foreach \i in {1,...,7} { \path[semithick] (\i) edge (\the\numexpr\i+1\relax); }
    \path[semithick] (8) edge (1);
    \path[ultra thick] (1) edge (4) (5) edge (8);
  \end{tikzpicture}
  \caption{The graph $H$ of Lemma~\ref{lem:parallel_edges} for $m=8$, $a=4$.}
  \label{fig:parallel_chords}
\end{figure}
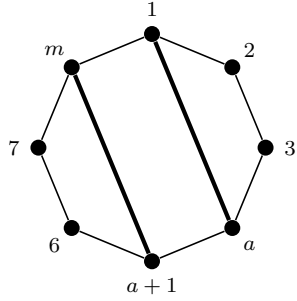

\begin{lemma}\label{lem:parallel_edges}
  Let $3 \leq a \leq m-3$, and let $H$ be the graph obtained from the $m$-cycle $C_m$, whose vertices are labeled $1, 2, \ldots, m$, by adding two chords $\{1,a\}$ and $\{a+1,m\}$. Then $\KG(G \sqcup C_{m-1},k)$ admits a homomorphism to $\KG(G \sqcup H, k)$, for every graph $G$.
\end{lemma}

\begin{proof}
  Assume that the vertices of $C_{m-1}$ are numbered $1, 2, \ldots, m-1$ in cyclic order, and consider the mapping $f:V(\KG(G \sqcup C_{m-1},k)) \to V(\KG(G \sqcup H,k))$ defined by
  \[
    f(A) =
      \begin{cases}
        (A \setminus \{1\}) \cup \{m\} & \text{ if } \{1,a\} \subseteq A \\
        A                            & \text{ otherwise.}
      \end{cases}
  \]
  We will show that $f$ defines a homomorphism from $\KG(G \sqcup C_{m-1},k)$ to $\KG(G \sqcup H, k)$.

  Since $A \subseteq V(G) \cup V(C_{m-1})$, it follows that $m \notin A$, so $|f(A)| = |A| = k$. Next, we show that $f(A)$ is an independent set of $G \sqcup H$. We consider two cases.

  If $\{1,a\} \not \subseteq A$, then $f(A) = A$. Therefore, $f(A)$ does not induce any edge of $C_m$. Since $m \notin f(A)$, $f(A)$ does not induce the edge $\{a+1,m\}$. Finally, since $\{1,a\} \not \subseteq A = f(A)$, $f(A)$ does not induce the edge $\{1,a\}$. This shows that $f(A)$ is an independent set of $G \sqcup H$.

  If $\{1,a\} \subseteq A$, then $f(A) = (A \setminus \{1\}) \cup \{m\}$. Since $1\in A$ we know $m-1\notin A$, and so $f(A)$ does not induce any edge of $C_m$. The set $f(A)$ does not contain the element $1$, so it does not induce the edge $\{1,a\}$. Since $a \in A$ and $A$ is independent in $C_{m-1}$, $a+1 \notin A$, so $a+1 \notin f(A)$. Therefore $f(A)$ does not induce the edge $\{a+1,m\}$, so it is an independent set of $H$.

  Finally, it remains to show that $f$ maps edges of $\KG(G \sqcup C_{m-1},k)$ to edges of $\KG(G \sqcup H,k)$. Fix an edge $AB$ of $\KG(G \sqcup C_{m-1},k)$. At most one of $A$ and $B$ contains $\{1,a\}$. If neither does, then $f(A) \cap f(B) = A \cap B = \emptyset$. Otherwise, assume without loss of generality that $\{1,a\} \subseteq A$. Then $f(B) = B$ and
  \[
    f(A) \cap f(B) = \bigl((A \cap B) \setminus (\{1\} \cap B)\bigr) \cup (\{m\} \cap B) = \emptyset,
  \]
  since $m \notin B$. We have shown that $f$ is a homomorphism from $\KG(G \sqcup C_{m-1},k)$ to $\KG(G \sqcup H, k)$, as required.
\end{proof}

Theorem~\ref{thm:chi_bound} is a straightforward consequence of Lemma~\ref{lem:parallel_edges}.

\begin{proof}[Proof of Theorem~\ref{thm:chi_bound}]
  We argue by induction on $c$. The case $c=1$ corresponds to Schrijver's theorem.

  Now assume that $c \geq 2$. Put $G = G' \sqcup D_1 \sqcup D_2$, where $D_1$ and $D_2$ are two cycles, and $G'$ is a (possibly empty) disjoint collection of $c-2$ cycles. Without loss of generality assume the cycles $D_1$ and $D_2$ are labeled $1, 2, \ldots, a$ and $a+1, a+2, \ldots, m$, respectively, in cyclic order; note that $3 \leq a \leq m-3$. By adding the edges $\{a,a+1\}$ and $\{1,m\}$ to $D_1 \sqcup D_2$, we obtain the graph $H$ of Lemma~\ref{lem:parallel_edges}. Denote the cycle $1, 2, \ldots, m-1, 1$ by $C$.

  Observe that $\KG(G' \sqcup H,k) \subseteq \KG(G,k)$ since $G \subseteq G' \sqcup H$. Hence by Lemma~\ref{lem:parallel_edges}, there is a homomorphism from $\KG(G' \sqcup C,k)$ to $\KG(G'\sqcup H,k)$. Moreover, $G' \sqcup C$ is a disjoint union of $c-1$ cycles of total length $n-1 \geq 2k+(c-1)-1$. By the induction hypothesis,
  \begin{align*}
    \chi(\KG(G,k)) &\geq \chi(\KG(G' \sqcup H,k)) \\
                   &\geq \chi(\KG(G'\sqcup C,k)) \\
                   &\geq (n-1)-2k+3-(c-1) \\
                   &= n-2k+3-c,
  \end{align*}
  as required.
\end{proof}

The next result shows that the bound of Theorem~\ref{thm:chi_bound} is attained for every number of cycles.

\begin{proposition}
  Let $2 \leq c \leq k$ be integers. Then there exists a disjoint union of $c$ cycles $G$ on $n=3k$ vertices with $\chi(\KG(G,k)) \leq n-2k+3-c$.
\end{proposition}

\begin{proof}
  Let $G$ be the disjoint union of $c$ cycles $D_1, D_2, \dots, D_{c}$ where $D_1$ has $k-c+3$ vertices and, for $i \geq 2$, the cycle $D_i$ has $j_i$ vertices with $j_i$ odd and $\sum_{i=2}^c j_i = 2k+c-3$. Such a choice is possible: for example, it suffices to take $j_2 = 2k-2c+3$ and $j_i = 3$ for every $i \geq 3$.

  We claim that any independent $k$-set $A$ of $G$ intersects $D_1$. Indeed, each $D_i$ with $i \geq 2$ is an odd cycle, so $|A \cap D_i| \leq (j_i-1)/2$, and therefore
  \[
    |A \setminus V(D_1)| \leq \sum_{i=2}^c \tfrac{j_i-1}2 = \tfrac12(2k+c-3-(c-1)) = k-1.
  \]
  The $|V(D_1)|$ stars centered at the vertices of $D_1$ therefore cover $\KG(G,k)$, so $\chi(\KG(G,k)) \leq |V(D_1)| = k+3-c = n-2k+3-c$, as required.
\end{proof}

\section{Cycles plus triangles}\label{sec:cycles+triangles}

A \emph{cycle-plus-triangles graph} is a $4$-regular multigraph on $3k$ vertices whose edge set can be decomposed into a Hamiltonian cycle and $k$ vertex-disjoint triangles. Du, Hsu and Hwang~\cite{DHH93} conjectured that every cycle-plus-triangles graph has an independent set of size $k$. Erd\H os made the stronger conjecture that such graph are $3$-colorable: its vertices can be partitioned into three independent $k$-sets.

Erd\H os's conjecture was proved by Fleischner and Stiebitz~\cite{FS92} via one of the first applications of the Alon--Tarsi polynomial method~\cite{AT92}. Soon thereafter an elementary---though by no means simple---inductive proof of Erd\H os's conjecture was found by Sachs~\cite{Sac93}. The original conjecture of Du, Hsu and Hwang is an almost immediate corollary of~\cite{AH2000} (see e.g. the survey~\cite{HaxellBCC}, Thm 4.14), in which the single cycle can be replaced by any disjoint union of cycles whose lengths are all divisible by three. The following stronger result for unions of cycles was shown by Aharoni, Holzman, Howard and Sprüssel~\cite{AHHS15}.

\begin{theorem}[Aharoni, Holzman, Howard and Sprüssel~\cite{AHHS15}]\label{thm:1mod3}
  Let $G$ be the disjoint union of cycles of total length $3k$, at most two of which have length $1 \pmod 3$, and $k$ vertex-disjoint triangles on the same vertex set. Then $G$ has an independent set of size $k$.
\end{theorem}

Haxell and Wdowinski~\cite{HW24} showed that this is best possible: they constructed, for infinitely many $k$, a union of three cycles of length $1 \pmod 3$ and total length $3k$ together with $k$ vertex-disjoint triangles on the same vertex set, having no independent set of size $k$.

A second proof of the Du--Hsu--Hwang conjecture was found by Aharoni et al.~\cite{AABCKLZ17}, using a clever application of Schrijver's theorem. In fact their argument showed the existence of two disjoint independent $k$-sets (instead of the three disjoint independent $k$-sets given by~\cite{FS92,Sac93}).

As an application of Theorem~\ref{thm:chi_bound}, we give a short proof of the following weakening of Theorem~\ref{thm:1mod3}. The argument is exactly that of Aharoni et al.~\cite{AABCKLZ17}, with Theorem~\ref{thm:chi_bound} in place of Schrijver's theorem.

\begin{corollary}\label{cor:2cycles+triangles}
  Let $G$ be the disjoint union of at most two cycles of total length $3k$ and $k$ vertex-disjoint triangles on the same vertex set. Then $G$ has an independent set of size $k$.
\end{corollary}

\begin{proof}
  Let $H$ be the spanning subgraph of $G$ consisting of $c \leq 2$ cycles, and let $T_1, T_2, \ldots, T_k$ be the vertex sets of the triangles of $G$. For each vertex $A$ of $\KG(H,k)$, color $A$ with color $i$ if $|A \cap T_i| \geq 2$; otherwise leave $A$ uncolored. This defines a $k$-coloring of an induced subgraph of $\KG(H,k)$. By Theorem~\ref{thm:chi_bound}, $\chi(\KG(H,k)) \geq 3k-2k+3-c \geq k+1$, so at least one vertex $A^*$ of $\KG(H,k)$ is uncolored. This means $|A^* \cap T_i| \leq 1$ for every $i \in [k]$, and therefore $|A^* \cap T_i| = 1$ for every $i \in [k]$. Hence, $A^*$ is an independent $k$-set in $G$.
\end{proof}

Note that if $H$ consists of more than one cycle, it cannot be $k$-tight: one of the cycles has length at most $3k/2 \leq 4k-4$ (since $k \geq 2$). So the lower bound of $k+1$ in the above proof is optimal for $H$ consisting of two cycles.

If $H$ consists of a single cycle (as in the original argument of Aharoni et al.~\cite{AABCKLZ17}), Schrijver's theorem gives the stronger bound $\chi(\KG(H,k)) \geq k+2$, and the uncolored subgraph of $\KG(H,k)$ contains at least one edge. This leads to the stronger conclusion mentioned above, namely, that there are \emph{two} disjoint independent $k$-sets.

If $H$ consists of three or more cycles, the bound coming from Theorem~\ref{thm:chi_bound} is too weak to guarantee the existence of any uncolored vertices (and therefore, of any independent $k$-set). This is consistent with Haxell and Wdowinski's construction~\cite{HW24}.

\smallskip
\noindent{\textbf{Declaration.}} No AI tools were used in any way in this paper.

\bibliographystyle{plain}
\bibliography{g-stable-kneser}

@preamble{
   "\def\soft#1{\leavevmode\setbox0=\hbox{h}\dimen7=\ht0\advance
    \dimen7 by-1ex\relax\if t#1\relax\rlap{\raise.6\dimen7
    \hbox{\kern.3ex\char'47}}#1\relax\else\if T#1\relax
    \rlap{\raise.5\dimen7\hbox{\kern1.3ex\char'47}}#1\relax
    \else\if d#1\relax\rlap{\raise.5\dimen7\hbox{\kern.9ex
    \char'47}}#1\relax\else\if D#1\relax\rlap{\raise.5\dimen7
    \hbox{\kern1.4ex\char'47}}#1\relax\else\if l#1\relax
    \rlap{\raise.5\dimen7\hbox{\kern.4ex\char'47}}#1\relax
    \else\if L#1\relax\rlap{\raise.5\dimen7\hbox{\kern.7ex
    \char'47}}#1\relax\else\message{accent \string\soft
    \space #1 not defined!}#1\relax\fi\fi\fi\fi\fi\fi} "
}

@incollection {AABCKLZ17,
    AUTHOR = {Aharoni, R. and Alon, N. and Berger, E. and Chudnovsky,
              M. and Kotlar, D. and Loebl, M. and Ziv, R.},
     TITLE = {Fair representation by independent sets},
 BOOKTITLE = {A journey through discrete mathematics},
     PAGES = {31--58},
 PUBLISHER = {Springer, Cham},
      YEAR = {2017},
      ISBN = {978-3-319-44478-9; 978-3-319-44479-6},
}

@article {AH2000,
    AUTHOR = {Aharoni, R. and Haxell, P.},
     TITLE = {Hall's Theorem for hypergraphs},
   JOURNAL = {J. Graph Th.},
  FJOURNAL = {Journal of Graph Theory},
    VOLUME = {35},
      YEAR = {2000},
    NUMBER = {2},
     PAGES = {83--88},
      ISSN = {},
       DOI = {},
       URL = {},
}

@article {AHHS15,
    AUTHOR = {Aharoni, R. and Holzman, R. and Howard, D. and
              Spr{\"u}ssel, P.},
     TITLE = {Cooperative colorings and independent systems of
              representatives},
   JOURNAL = {Electron. J. Combin.},
  FJOURNAL = {Electronic Journal of Combinatorics},
    VOLUME = {22},
      YEAR = {2015},
    NUMBER = {2},
     PAGES = {Paper 2.27, 14},
      ISSN = {1077-8926},
       DOI = {10.37236/2488},
       URL = {https://doi.org/10.37236/2488},
}

@article {AT92,
    AUTHOR = {Alon, N. and Tarsi, M.},
     TITLE = {Colorings and orientations of graphs},
   JOURNAL = {Combinatorica},
  FJOURNAL = {Combinatorica. An International Journal on Combinatorics and
              the Theory of Computing},
    VOLUME = {12},
      YEAR = {1992},
    NUMBER = {2},
     PAGES = {125--134},
      ISSN = {0209-9683},
       DOI = {10.1007/BF01204715},
       URL = {https://doi.org/10.1007/BF01204715},
}

@article {Che15,
    AUTHOR = {Chen, P.},
     TITLE = {On the multichromatic number of {$s$}-stable {K}neser graphs},
   JOURNAL = {J. Graph Theory},
  FJOURNAL = {Journal of Graph Theory},
    VOLUME = {79},
      YEAR = {2015},
    NUMBER = {3},
     PAGES = {233--248},
      ISSN = {0364-9024,1097-0118},
       DOI = {10.1002/jgt.21826},
       URL = {https://doi.org/10.1002/jgt.21826},
}

@article {DMM21,
    AUTHOR = {Daneshpajouh, H. and Meunier, F. and Mizrahi, G.},
     TITLE = {Colorings of complements of line graphs},
   JOURNAL = {J. Graph Theory},
    VOLUME = {98},
      YEAR = {2021},
    NUMBER = {2},
     PAGES = {216--233},
      ISSN = {0364-9024,1097-0118},
       DOI = {10.1002/jgt.22689},
       URL = {https://doi.org/10.1002/jgt.22689},
}

@article {DHH93,
    AUTHOR = {Du, D. Z. and Hsu, D. F. and Hwang, F. K.},
     TITLE = {The {H}amiltonian property of consecutive-{$d$} digraphs},
      NOTE = {Graph-theoretic models in computer science, II (Las Cruces,
              NM, 1988--1990)},
   JOURNAL = {Math. Comput. Modelling},
  FJOURNAL = {Mathematical and Computer Modelling},
    VOLUME = {17},
      YEAR = {1993},
    NUMBER = {11},
     PAGES = {61--63},
      ISSN = {0895-7177},
       DOI = {10.1016/0895-7177(93)90253-U},
       URL = {https://doi.org/10.1016/0895-7177(93)90253-U},
}

@article {FS92,
    AUTHOR = {Fleischner, H. and Stiebitz, M.},
     TITLE = {A solution to a colouring problem of {P}. {E}rd{\H{o}}s},
      NOTE = {Special volume to mark the centennial of Julius Petersen's
              ``Die Theorie der regul\"aren Graphs'', Part II},
   JOURNAL = {Discrete Math.},
  FJOURNAL = {Discrete Mathematics},
    VOLUME = {101},
      YEAR = {1992},
    NUMBER = {1-3},
     PAGES = {39--48},
      ISSN = {0012-365X,1872-681X},
       DOI = {10.1016/0012-365X(92)90588-7},
       URL = {https://doi.org/10.1016/0012-365X(92)90588-7},
}

@article {Hav24,
    AUTHOR = {Haviv, I.},
     TITLE = {Fixed-parameter algorithms for the {K}neser and {S}chrijver
              problems},
   JOURNAL = {SIAM J. Comput.},
  FJOURNAL = {SIAM Journal on Computing},
    VOLUME = {53},
      YEAR = {2024},
    NUMBER = {2},
     PAGES = {287--314},
      ISSN = {0097-5397,1095-7111},
       DOI = {10.1137/23M1557076},
       URL = {https://doi.org/10.1137/23M1557076},
}

@incollection {HaxellBCC,
    AUTHOR = {Haxell, P.},
     TITLE = {Topological connectedness and independent sets in graphs},
 BOOKTITLE = {Surveys in Combinatorics 2019, LMS Lecture Notes Series 456},
     PAGES = {89--113},
 PUBLISHER = {Cambridge Univ. Press, Cambridge},
      YEAR = {2019},
      ISBN = {},
}

@article {HW24,
    AUTHOR = {Haxell, P. and Wdowinski, R.},
     TITLE = {Constructing graphs with no independent transversals},
   JOURNAL = {Electron. J. Combin.},
    VOLUME = {31},
      YEAR = {2024},
    NUMBER = {2},
     PAGES = {Paper No. 2.39, 28},
      ISSN = {1077-8926},
       DOI = {10.37236/12429},
       URL = {https://doi.org/10.37236/12429},
}

@unpublished {Jon12,
    AUTHOR = {Jonsson, J.},
     TITLE = {On the chromatic number of generalized stable Kneser graphs},
      NOTE = {Unpublished preprint, available at \url{https://people.kth.se/~jakobj/research.html}},
      YEAR = {2012}
}

@article {Kap43,
    AUTHOR = {Kaplansky, I.},
     TITLE = {Solution of the ``probl\`eme des m\'enages.''},
   JOURNAL = {Bull. Amer. Math. Soc.},
  FJOURNAL = {Bulletin of the American Mathematical Society},
    VOLUME = {49},
      YEAR = {1943},
     PAGES = {784--785},
      ISSN = {0002-9904},
       DOI = {10.1090/S0002-9904-1943-08035-4},
       URL = {https://doi.org/10.1090/S0002-9904-1943-08035-4},
}

@article {Kne55,
    AUTHOR = {Kneser, M.},
     TITLE = {Aufgabe 360},
   JOURNAL = {Jahresber. Deutsch. Math.-Verein.},
    VOLUME = {58},
      YEAR = {1955},
    NUMBER = {2. Abteilung},
     PAGES = {27},
}

@article {Lov78,
    AUTHOR = {Lov{\'a}sz, L.},
     TITLE = {Kneser's conjecture, chromatic number, and homotopy},
   JOURNAL = {J. Combin. Theory Ser. A},
    VOLUME = {25},
      YEAR = {1978},
    NUMBER = {3},
     PAGES = {319--324}
}

@book {Mat03,
    AUTHOR = {Matou{\v{s}}ek, J.},
     TITLE = {Using the {B}orsuk--{U}lam theorem},
    SERIES = {Universitext},
 PUBLISHER = {Springer-Verlag},
   ADDRESS = {Berlin},
      YEAR = {2003}
}

@article {Meu11,
    AUTHOR = {Meunier, F.},
     TITLE = {The chromatic number of almost stable {K}neser hypergraphs},
   JOURNAL = {J. Combin. Theory Ser. A},
  FJOURNAL = {Journal of Combinatorial Theory. Series A},
    VOLUME = {118},
      YEAR = {2011},
    NUMBER = {6},
     PAGES = {1820--1828},
       DOI = {10.1016/j.jcta.2011.02.010},
       URL = {https://doi.org/10.1016/j.jcta.2011.02.010},
}

@incollection {Sac93,
    AUTHOR = {Sachs, H.},
     TITLE = {Elementary proof of the cycle-plus-triangles theorem},
 BOOKTITLE = {Combinatorics, {P}aul {E}rd\H os is eighty, {V}ol.\ 1},
    SERIES = {Bolyai Soc. Math. Stud.},
     PAGES = {347--359},
 PUBLISHER = {J\'anos Bolyai Math. Soc., Budapest},
      YEAR = {1993},
      ISBN = {963-8022-74-4},
}

@article {Sch78,
    AUTHOR = {Schrijver, A.},
     TITLE = {Vertex-critical subgraphs of {K}neser graphs},
   JOURNAL = {Nieuw Arch. Wisk. (3)},
    VOLUME = {26},
      YEAR = {1978},
    NUMBER = {3},
     PAGES = {454--461}
}
\end{document}